\documentclass{article}
\usepackage[T2A]{fontenc}
\usepackage[cp1251]{inputenc}
\usepackage[english,russian]{babel}
\usepackage[tbtags]{amsmath}
\usepackage{amsfonts,amssymb,mathrsfs,amscd}
\usepackage{graphics}
\usepackage{euscript,textcomp,verbatim,fancyhdr,stmaryrd}
\usepackage{latexsym}
\usepackage[hyper]{msb-a}
\JournalName{}
\theoremstyle{plain}
\newtheorem{Thm}{Theorem} 
\newtheorem{Lem}{Lemma} 
\newtheorem{Cor}{Corollary} 

\theoremstyle{definition}
\newtheorem{proof}{Proof} 
\newtheorem{Def}{Definition} 
\newtheorem{Exa}{Example} 

\newcommand{\RR}{\mathbb R}

\newcommand{\ZZ}{\mathbb Z}
\newcommand{\QQ}{\mathbb Q}

\renewcommand{\tilde}{\widetilde}
\newcommand{\eps}{\varepsilon}
\renewcommand{\d}{\partial}
\renewcommand{\o}{{^\circ}} 

\newcommand{\const}{\mathrm{const}}
\newcommand{\Difff}{{\rm Diff}}

\newcommand{\Cal}{{\rm Cal}}

\renewcommand{\Im}{{\rm Im}}
\newcommand{\Hom}{{\rm Hom}}

\newcommand{\Sy}{{\cal G}} 
\newcommand{\tS}{\tilde{{\cal G}}} 
\newcommand{\cB}{{\cal B}} 
\newcommand{\cA}{{\cal A}} 
\newcommand{\cD}{{\cal D}} 
\renewcommand{\H}{{\cal H}} 
\newcommand{\cF}{{\rm Flux}}

\begin{document}

\title
{Helicity is the only invariant of incompressible flows whose derivative is continuous in $C^1$--topology}
\author{Elena A.~Kudryavtseva}
\address{Moscow State University}
\email{eakudr@mech.math.msu.su}
\date{}
\udk{}

\maketitle
\begin{fulltext}

Let $Q$ be a smooth compact orientable 3--manifold with smooth boundary $\d Q$. Let $\cB$ be the set of exact 2--forms $B\in\Omega^2(Q)$ such that $j_{\d Q}^*B=0$, where $j_{\d Q}:{\d Q}\to Q$ is the inclusion map.
The group $\cD=\Difff_0(Q)$ of self-diffeomorphisms of $Q$ isotopic to the identity acts on the set $\cB$ by $\cD\times\cB\to\cB$, $(h,B)\mapsto h^*B$. Let $\cB\o$ be the set of 2--forms $B\in\cB$ without zeros. We prove that every $\cD$--invariant functional $I:\cB\o\to\RR$ having a regular and continuous derivative with respect to the $C^1$--topology can be {\em locally} (and, if $Q=M\times S^1$ with $\d Q\ne\varnothing$, {\em globally} on the set of all 2--forms $B\in\cB\o$ admitting a cross-section isotopic to $M\times\{*\}$) expressed in terms of the helicity.



\markright{Helicity is the only invariant of incompressible flows}

\footnotetext[0]{This work was supported by the Russian Foundation for Basic Research (grant no.\ 15–01–06302-a) and
the program
``Leading Scientific Schools'' (grant no.\ NSh-581.2014.1).
}

\section
 {Examples of $\cD$--invariant functionals on the set $\cB$ of exact incompressible flows}

 Let us define the flax and helicity.

\begin{Def} \label {def:helic}
(A) Let $\Pi\in C_2(Q)$ be a 2--chain in $Q$ with boundary $\d\Pi\in C_1(\d Q)$. Let $\llbracket\Pi\rrbracket\in\varkappa$ be the projection of the relative homology class $[\Pi]\in H_2(Q,\d Q;\QQ)$ into the quotient space
 \begin{equation} \label {eq:varkappa}
\varkappa:= H_2(Q,\d Q;\QQ)/\Im\,{p_*} \cong \ker({j_{\d Q}})_* = \d(\varkappa),
 \end{equation}
where ${p_*}: H_2(Q;\QQ){\to} H_2(Q,\d Q;\QQ)$, $({j_{\d Q}})_*:H_1(\d Q;\QQ){\to} H_1(Q;\QQ).$ The {\em flux} on the set $\cB$ is defined by the formula $\cF:\cB\to\Hom_\QQ(\varkappa,\RR)$, $\cF(B){\llbracket\Pi\rrbracket}:=\int_\Pi B$.

(B) Let $\varkappa^\perp\subseteq H_1(\d Q;\QQ)$ be an arbitrary vector space such that $H_1(\d Q;\QQ)=\varkappa^\perp\oplus\ker({j_{\d Q}})_*$. The {\em helicity} on the set $\cB$ with respect to the subspace $\varkappa^\perp$ is defined by the formula
 \begin{equation} \label {eq:helic:d}
\H_{\varkappa^\perp}:\cB\to\RR, \qquad \H_{\varkappa^\perp}(B):=\int_Q B\wedge A, \quad B \in\cB,
 \end{equation}
where $A\in\Omega^1(Q)$ is any 1--form on $Q$ such that $dA=B$ and $\oint_\gamma A=0$ for any loop $\gamma:S^1\to\d Q$ of homology class $[\gamma]\in \varkappa^\perp$. 
\end{Def}

It can readily be seen that $\cF$ and $\H_{\varkappa^\perp}$ are well defined and $\cD$--invariant.

\begin{Exa} \label {exa:1}
Let $Q=M\times S^1$ where $M$ is a compact smooth oriented surface with nonempty boundary. Then there exists an exact positive area form 
$\omega\in\Omega^2(M)$ on $M$. Let $B\in\cB\o$ be such that $j_{M\times\{0\}}^*B=\omega$ (i.e.\ $M\times\{0\}$ is a cross-section for the
2--form $B$), $S_1,\dots,S_d\subseteq\d M$ be the boundary circles, $\pi_M:Q\to M$ and $\pi_{S^1}:Q\to S^1=\RR/\ZZ$ the projections. Since $B$ is exact, due to \cite {Jen,Ban1978} there exist a diffeomorphism $\psi\in\cD$ and a function $H\in C^\infty(Q)$ such that $B=\psi^*B_{\omega,H}$, $H|_{S_1\times S^1}\equiv0$ and $H|_{S_i\times\{t\}}=\const=:h_i(t)$ for $t\in S^1$ and  $i\in\{2,\dots,d\}$, where 
 \begin{equation} \label {eq:exa:1}
B_{\omega,H}:=\pi_M^*\omega-dH\wedge d\pi_{S^1} .
 \end{equation}
For $i\in\{1,\dots,d\}$, we fix a point $x_i\in S_i$ and paths $\gamma_{ik}\subset M$ from $x_i$ to the points $x_k$, $k\ne i$. Put $\Pi_{ik}:=\gamma_{ik}\times S^1$. The flux and helicity of the 2--form $B$ are as follows:

(A) the classes $\llbracket M\times\{0\}\rrbracket,\llbracket\Pi_{12}\rrbracket,\dots,\llbracket\Pi_{1d}\rrbracket\in\varkappa$ form a basis of $\varkappa$, and the flux
$$
\cF(B)\llbracket M\times\{0\}\rrbracket=\int_M\omega, \qquad 
\cF(B)\llbracket\Pi_{\ell k}\rrbracket=\int_{S^1}h_\ell(t)dt-\int_{S^1}h_k(t)dt;
$$

(B) for any $\ell,k\in\{1,\dots,d\}$, we take as $\varkappa^\perp$ the subspace $\varkappa^\perp_{\ell k}\subseteq H_1(\d Q;\QQ)$ generated by the homology classes of the loops $\gamma_k:=\{x_k\}\times S^1$ and $S_i\times\{0\}$, $i\in\{1,\dots,d\}\setminus\{\ell\}$. By the Stokes formula, the helicity is
 $$
\H_{\varkappa^\perp_{11}}(B) =-2\int_QH (\pi_M^*\omega)\wedge (d\pi_{S^1}) =-2\Cal_\omega(\tilde\varphi), 
 $$
 $$
\H_{\varkappa^\perp_{\ell k}}(B) - \H_{\varkappa^\perp_{\ell\ell}}(B)
= \frac12 (\H_{\varkappa^\perp_{\ell\ell}}(B) - \H_{\varkappa^\perp_{kk}}(B)) 
= \cF(B)\llbracket M\times\{0\}\rrbracket \, \cF(B)\llbracket\Pi_{\ell k}\rrbracket
 $$
for $k\ne\ell$. Here $\Cal_\omega:\tS_\omega\to\RR$ is the Calabi homomorphism \cite{Calabi70,K:magn:f}, 
$\tS_\omega$ is a universal cover of the group $\Sy_\omega$ of symplectomorphisms of the surface $(M,\omega)$ with $C^\infty$-topology, the element $\tilde\varphi\in\tS_\omega$ is determined via the 2--form (\ref {eq:exa:1}).
\end{Exa}

\section{Differentiable functionals on $\cB''\subseteq\cB$ and $\H_{\varkappa^\perp}|_{\cB''}$--connected subsets of $\cB''$}

Denote by $\cA$ (respectively $\cA_{\varkappa^\perp}$) the set of 1--forms $A\in\Omega^1(Q)$ such that $\oint_\gamma A=0$ for any loop $\gamma:S^1\to\d Q$ (respectively any loop of homology class $[\gamma]\in \varkappa^\perp$, cf.\ Definition \ref {def:helic} (B)). Let  $\Phi\in\Hom_\QQ(\varkappa,\RR)$ be a linear function on (\ref {eq:varkappa}) and 
\begin{equation} \label {eq:cB'}
\mbox{either } \ \cB':= \cF^{-1}(\Phi) \ 
\mbox{ and }\ \cA':=\cA, \quad 
\mbox{ or } \ \cB':=\cB \ 
\mbox{ and }\ \cA':=\cA_{\varkappa^\perp}.
\end{equation}
Clearly, $\cB'$ and $\cA'$ are $\cD$--invariant, and the set $\cB'-B$ is $C^0$--open in the vector space $d\cA'$ for any $B\in\cB'$.

Let $\cB''\subseteq\cB'$ be a $C^1$--open subset (for example, $\cB''=\cB'\cap\cB\o$).

\begin{Def} \label {def:diff}
A functional $I:\cB''\to\RR$ is said to be {\em differentiable} at a point $B\in\cB''$ if there is a linear functional $D_BI:\cA'\to\RR$ (referred to as the {\em derivative} of the functional $I$ at the point $B$) such that, for every smooth family of 1--forms $A_u\in\cA'$, $-\eps<u<\eps$, for which $dA_0=0$, we have
$$
\left.\frac{d}{du}\right|_{u=0}I(B+dA_u)=D_BI\left(\left.\frac{d}{du}\right|_{t=0}A_u\right).
$$ 
The derivative $D_BI$ is said to be {\em regular} if it is a regular element of the dual space $(\cA')^*$, i.e., $D_BI(A')=\int_{Q}K_I(B)\wedge A'$, $A'\in\cA'$, for some measurable 2--form $K_I(B)$ on $Q$; this 2--form is referred to as the {\em density} of the functional $D_BI$. Let $I$ be differentiable everywhere on $\cB''$ and have a regular derivative. We say that the derivative is {\em continuous with respect to the $C^k$--topology on $\cB''$} (for $k\ge0$) if its density is a continuous map  $K_I:\cB''\to\Gamma^0(\Lambda^2T^*Q)$ 
with respect to the $C^k$-topology on $\cB''$ and $C^0$--topology on $\Gamma^0(\Lambda^2T^*Q)$.
\end{Def}

\begin{Exa} \label {exa:2}
The helicity (\ref {eq:helic:d}) is differentiable everywhere on $\cB$ (and hence on $\cB'\subseteq\cB$). Its derivative is regular and has density $K_{\H_{\varkappa^\perp}}(B)=B$. Therefore, its derivative is continuous with respect to every topology on $\cB$.
\end{Exa}

\begin{Def} \label {def:diff:}
A subset $\cB'''\subseteq\cB''$ is said to be {\em $\H_{\varkappa^\perp}|_{\cB''}$--connected} if every pair of its elements with the same value $c\in\RR$ of helicity $\H_{\varkappa^\perp}$ can be joined by a piecewise smooth path in $\H_{\varkappa^\perp}^{-1}(c)\cap\cB''$.
 \end{Def}

\begin{Lem} [(examples of $\H_{\varkappa^\perp}|_{\cB''}$--connected subsets)] \label {lem:1}
{\rm(A)} If $\cB'=\cB$ in {\rm(\ref {eq:cB'})} and $B\in{\cB''}$ is such a 2-form that $\H_{\varkappa^\perp}(B)\ne0$ then every sufficiently small $C^1$--neighbourhood $\cB'''$ of the 2--form $B$ in $\cB''$ is $\H_{\varkappa^\perp}|_{\cB''}$--connected.

{\rm(B)} Let $Q=M\times S^1$ as in Example {\rm\ref {exa:1}}. Let {\rm(\ref {eq:cB'})}, and let $\cB''\subseteq\cB'\cap\cB\o$ be the set of all 2--forms $\psi^*B_{\omega,H}$ where $\psi\in\cD$ and $B_{\omega,H}\in\cB'$ as in {\rm(\ref {eq:exa:1})}. Then the whole set $\cB''$ is $\H_{\varkappa^\perp}|_{\cB''}$--connected. 
\end{Lem}

\begin{proof} (A) Let us join any pair $B_0,B_1\in\H_{\varkappa^\perp}^{-1}(c)\cap\cB'''$ by the path $\{((1-u)B_0+uB_1)\sqrt{c/\H_{\varkappa^\perp}((1-u)B_0+uB_1)}\}_{u\in[0,1]}$ in $\H_{\varkappa^\perp}^{-1}(c)$. This path is contained in $\cB''$, since $\cB''$ is $C^1$--open and $\H_{\varkappa^\perp}$ is $C^0$--continuous.

(B) The set $\cB''$ is $C^0$--open (and hence $C^1$--open too) in $\cB'$. In view of \cite{Jen}, every pair of 2--forms from $\H_{\varkappa^\perp}^{-1}(c)\cap\cB''$ has the form $\psi_0^*B_{\omega,H_0},\psi_1^*B_{\lambda\omega,H_1}$ for some $\psi_0,\psi_1\in\cD$, $H_0,H_1\in C^\infty(Q)$, $\lambda\in\RR_{>0}$ and a positive area form $\omega\in\Omega^2(M)$ such that $B_{\omega,H_0},B_{\lambda\omega,H_1}\in\H_{\varkappa^\perp}^{-1}(c)$. Since every pair $\psi_0,\psi_1\in\cD$ can be joined by a piecewise smooth path $\{\psi_u\}_{u\in[0,1]}$ in the group $\cD$ \cite {Jen}, we obtain the path $\{\psi_u^*B_{(1-u+u\lambda)\omega,a(u)H_0+uH_1}\}$ in $\H_{\varkappa^\perp}^{-1}(c)\cap\cB''$, which joins our 2--forms $\psi_0^*B_{\omega,H_0}$ and $\psi_1^*B_{\lambda\omega,H_1}$, where $a(u):=1/(1-u+u\lambda)-u/\lambda$, $0\le u\le 1$.
 \end{proof}

\section{Main result}

Our result was announced in \cite{K2015}. It is similar to the results of \cite {K:magn:f,Serre84}. We prove it by the technique of the paper \cite {K:magn:f}.

\begin{Thm} 
Let $\varkappa,\varkappa^\perp,\cB'$ be defined as in {\rm(\ref {eq:varkappa}), (\ref {eq:helic:d})}, {\rm(\ref {eq:cB'})}.
Let $I:\cB'\to\RR$ be a $\cD$--invariant functional on $\cB'$ differentiable on a $C^1$--open subset $\cB''\subseteq\cB'\cap\cB\o$ and having a regular and continuous derivative with respect to the $C^1$--topology on $\cB''$. Then the restriction of this functional to every $\H_{\varkappa^\perp}|_{\cB''}$--connected subset $\cB'''\subseteq\cB''$ {\rm(}e.g.\ one from Lemma {\rm\ref {lem:1})}, can be expressed using the helicity, i.e., $I|_{\cB'''}=h\circ\H_{\varkappa^\perp}|_{\cB'''}$ for some function $h:\RR\to\RR$.
\end{Thm}

\begin{Lem} \label {lem:2}
Suppose that $I:\cB'\to\RR$ is a $\cD$--invariant functional differentiable at a point $B\in\cB''$. Then $D_BI(A')=D_BI(\psi^*A')$ for every 1--form $A'\in\cA'$ and for every diffeomorphism $\psi\in\cD$ such that $\psi^*B=B$.
\end{Lem}

\begin{proof}
In the notation of Definition \ref {def:diff}, we have $I(B+dA_t)=I(B+\psi^*dA_t)=I(B+d\tilde A_t)$ for $\tilde A_t:=\psi^*A_t$, and $d\tilde A_0=0$. Differentiating the relation $I(B+dA_t)=I(B+d\tilde A_t)$ proved above with respect to $t$ at $t=0$, we obtain the desired equality for $A'=\frac{d}{dt}|_{t=0}A_t$. This completes the proof of the lemma.
 \end{proof}

\begin{Cor} 
Suppose that, under the assumptions of Lemma {\rm\ref {lem:2}}, the derivative $D_BI$ at the point $B$ is regular and its density $K_I(B)$ is a continuous 2--form on $Q$. Then $K_I(B)=\psi^* K_I(B)$. If the 2--form $B$ has no zeros {\rm(}i.e.\ $B\in\cB\o${\rm)} then  $K_I(B)=\lambda_I(B)\, B$ for some function $\lambda_I(B)\in C(Q)$ that is constant on every integral curve of the field of kernels of the 2--form $B$.
\end{Cor}

\begin{proof}
The equality of Lemma \ref {lem:2} can be represented in the form $\int_{Q}K_I(B)\wedge A'=\int_{Q}K_I(B)\wedge \psi^*A'$. But the left-hand side of this equality equals $\int_{Q}(\psi^*K_I(B))\wedge \psi^*A'$. Thus $\int_{Q}(\psi^*K_I(B)-K_I(B))\wedge \psi^*A'=0$. Since the 1--form $A'\in\cA'$ is arbitrary (and, in particular, one can take $A'\in\cA\subseteq\cA'$ supported in an arbitrarily small neighbourhood of an inner point of $Q$) and the 2--form $K_I(B)$ is continuous on $Q$, we obtain the first desired equality $K_I(B)=\psi^* K_I(B)$.

If $B$ has no zeros then any point of $Q\setminus\d Q$ has a neighbourhood $U\approx(-\eps,\eps)^3$ in $Q\setminus\d Q$ with regular coordinates $w=(x,y,z)\in(-\eps,\eps)^3$ such that $B|_U=dx\wedge dy$. By the equality proved above, the 2--form $K_I(B)|_U=:L(w)dy\wedge dz+M(w)dz\wedge dx+N(w)dx\wedge dy$ is invariant under transformations $\psi:w\mapsto\tilde w:=w+(0,0,f(w))$, for every function $f\in C^\infty((-\eps,\eps)^3)$ with a compact support and $\d f/\d z>-1$, where $L,M,N\in C((-\eps,\eps)^3)$. We observe that
 $$
\left(\psi^*K_I(B)- K_I(B)\right)|_U = (L(\tilde w)-L(w))dy\wedge dz +(M(\tilde w)-M(w))dz\wedge dx +
$$
$$
+(N(\tilde w)-N(w))dx\wedge dy
+(L(\tilde w)dy-M(\tilde w)dx) \wedge df(w) .
 $$
For any point $w_0=(x_0,y_0,z_0)\in(-\eps,\eps)^3$, take a function $f_1$ such that $f_1(w)=x-x_0$ in a neighbourhood of the point $w_0$. Since $f_1(w_0)=0$, we have $0=(\psi^*K_I(B)- K_I(B))|_{w=w_0}=L(w_0)dy\wedge dx$, thus $L(w_0)=0$.
One shows in a similar way that $M(w_0)=0$. Therefore $L\equiv M\equiv 0$, thus $K_I(B)|_U=N(w)\, B|_U$. 

Thus the relation $K_I(B)=\lambda_I(B)\, B$ is proved on $Q\setminus\d Q$. Hence it holds on $Q$ too, since $B$ has no zeros. 
Since $0=(\psi^*K_I(B)- K_I(B))|_U= (N(\tilde w)-N(w))B|_U$ and $f$ is arbitrary, the function $\lambda_I(B)|_U=N(x,y,z)$ does not depend on $z$, i.e.\ it is constant on integral curves of the field $\ker B$.
This completes the proof of the corollary.
\end{proof}

\begin{proof}[of the theorem]
%
Let us show that, for every 2--form $B\in\cB''$, the function $\lambda_I(B)\in C(Q)$ from the corollary is constant on $Q$. 
Recall that a vector field is said to be {\em topologically transitive} if it has an everywhere dense integral curve. Let $\mu$ be a volume form on $Q$. By a result \cite{Bessa}, the vector field $\overline B$ with $i_{\overline B}\mu=B$ can be $C^1$--approximated by a sequence $\{\overline B_n\}$ of topologically transitive divergence-free $C^1$--vector fields on $Q$ tangent to $\d Q$. Since $\overline B$ is exact, it follows from the proof in \cite{Bessa} that $\overline B_n$ can be chosen to be exact and having the same flux as $\overline B$. Every such a 2--form $B_n$ can be $C^1$--appro\-xi\-ma\-ted by a smooth 2--form $B_n^s\in\cB'$. Since $\cB''$ is $C^1$--open in $\cB'$, we have $B_n^s\in\cB''$. 

Thus, for every pair of points $w,w'\in Q$, there are sequences $\{B_n^s\}$ in $\cB''$ and $\{w_n\},\{w_n'\}$ in $Q$ such that $B_n^s\to B$ in the $C^1$--topology, $w_n\to w$, $w_n'\to w'$, and every pair $w_n,w_n'$ is contained in one integral curve of $\overline B_n^s$. Since $\lambda_I(B_n^s)$ is constant on every integral curve of $\overline B_n^s$, we have $\lambda_I(B_n^s)(w_n)=\lambda_I(B_n^s)(w_n')$. Since the map $K_I:\cB''\to\Gamma^0(\Lambda^2T^*Q)$ is continuous, it follows from the corollary and absence of zeros of every $B_1\in\cB''$ that the map $\lambda_I:\cB''\to C(Q)$ is continuous with respect to the $C^1$-topology on $\cB''$ and $C^0$-topology on $C(Q)$. Thus $\lambda_I(B)(w)$ and $\lambda_I(B)(w')$ are limits of $\{\lambda_I(B_n^s)(w_n)=\lambda_I(B_n^s)(w_n')\}$, so they coincide, thus $\lambda_I(B)=\const$.

We have $D_BI(A')=\lambda_I(B)\int_{Q}B\wedge A'=\lambda_I(B)\,D_B\H_{\varkappa^\perp}(A')$. Thus, the functional $I|_{\cB''}$ is locally constant on every level set of the helicity. Since the set $\cB'''\subseteq\cB''$ is $\H_{\varkappa^\perp}|_{\cB''}$--connected (cf.\ Definition \ref {def:diff:}), the functional $I|_{\cB'''}$ is constant on $\H_{\varkappa^\perp}^{-1}(c)\cap\cB'''$, i.e., it is equal to some constant $h(c)$ depending only on $c\in\RR$ (for some function $h:\RR\to\RR$). Thus, $I(B)=h(\H_{\varkappa^\perp}(B))$ for every $B\in\cB'''$. This completes the proof of the theorem.
\end{proof}

The author wishes to express gratitude  to G.\ Hornig for stating the problem and useful discussions, L.V.\ Polterovich and D.\ Peralta-Salas for their interest, useful discussions and indicating the papers \cite {BonCro2004} and \cite {Bessa}.

\end{fulltext}

\renewcommand {\refname}{References}

\end{document}